\documentclass[11pt]{tran-l}
\usepackage{tikz}
\usepackage[centertags]{amsmath}
\usepackage{amsfonts}
\usepackage{amssymb}
\usepackage{amsthm}
\usepackage{graphics, graphicx}
\usepackage{newlfont}
\usepackage{tabularx}

\theoremstyle{plain}
\newtheorem{thm}{Theorem}[section]
\newtheorem{cor}[thm]{Corollary}

\newtheorem{prop}[thm]{Proposition}
\theoremstyle{definition}
\newtheorem{defn}{Definition}[section]
\theoremstyle{remark}
\newtheorem{rem}{Remark}[section]
\theoremstyle{Example}

\numberwithin{equation}{section}

\begin{document}
\title{Clique Polynomials and Chordal Graphs
}
\author{Hossein Teimoori Faal}

\address{Department of Mathematics and Computer Science, 
Allameh Tabatabai University, Tehran, Iran}

\email{hossein.teimoori@atu.ac.ir\\ hossein.teimoori@gmail.com}

\maketitle
%%% ----------------------------------------------------------------------

\begin{abstract}
The ordinary generating function of the number 
of complete subgraphs 
of $G$ is called a clique polynomial of $G$ and is denoted by $C(G,x)$. A real root of $C(G,x)$ is called a clique root of the graph 
$G$. Hajiabolhasan and Mehrabadi showed that the 
clique polynomial has always a real root in the interval 
$[-1,0)$. Moreover, they showed that the class of triangle-free
graphs has only clique roots. Here, we generalize their 
result by showing that the class of $K_4$-free chordal 
graphs has also only clique roots. Moreover, we show that this class has always a clique root $-1$. We finally conclude the paper
with several important questions and conjectures.

\end{abstract}

%%% ----------------------------------------------------------------------
\section{Introduction and Motivation}
For a finite and simple graph $G= (V, E)$, a complete subgraph of $G$ on $k$ vertices is called 
a $k$-\emph{clique}. For a subset $U \subseteq V(G)$, the subgraph induced on $U$ will be denoted
by $G[U]$. We recall that an edge which joins two vertices of a cycle but is not itself an edge of the cycle is called
a \emph{chord} of the cycle. A graph
is 
\emph{chordal} 
if each cycle of length at least four has a chord. We also recall that
the \emph{Clique Polynomial} $[1]$ of the graph $G$, denote it by 
$C(G,x)$ is defined, as follows
%\begin{defn}
\begin{equation}
C(G,x):= 1 + \sum_{\emptyset \neq U \subseteq V(G):G[U]~is~a~clique}x^{\vert U\vert}, 
\end{equation} 
%\end{defn}
or equivalently 
\begin{equation}
C(G,x):= 1 + \sum_{k=1}^{\omega(G)} c_{k}(G), 
\end{equation}
in which $c_{k}(G)$ is the number of 
$k$-cliques in $G$ and $\omega(G)$ is the size of 
the largest clique of $G$. 
\\ 
Hajiabolhasan and Mehrabadi $[1]$ showed that the
clique polynomial of any simple graph has always a real root
in the interval $[-1,0)$. 
We will call it the \emph{clique root} of $G$. 
They also 
showed that class of triangle - free graphs has 
only clique roots. 
\\ 
In the recent years, proving that a particular class of graphs polynomial has 
only real roots has attracted the attention of many researchers working 
in the area of algebraic graph theory. This is partly because of the 
interesting properties of polynomials with only real roots 
like being \emph{unimodal}. 
\\
Our main goal here is to contribute in this line of research by 
extending the above result for the class of 
$K_4$ - free chordal graphs. More precisely, we will prove 
the following.  
\begin{thm}\label{keythm1}
%[All Clique Roots Property For $K_4$ - free chordal graphs]
The class of $K_4$ - free chordal graphs has only clique 
roots. In particular, they have always a clique root $-1$. 
\end{thm}
We will also give the following immediate corollary of the above theorem, which is indeed a new algebraic proof of 
Turan's graph theorem for planar $K_{4}$ - free graphs.  
\begin{cor}\label{keycor1}
If $G$ is a $K_{4}$ - free 
connected planar graph 
with $n$ vertices and $m$ edges, then we have
\begin{equation}
m \leq \frac{n^2}{3}\nonumber 
\end{equation}
\end{cor}
\section{Chordal Graphs and Clique Polynomials}
In this section, we investigate the important class of \emph{chordal graphs}. This class of graphs 
is very important in computer science, specially from the computational complexity view point. Many hard 
problems in general graphs have easy solutions in the class of chordal graphs. As we will see, the 
clique polynomial of chordal graphs can give us important insights into the structure of these graphs. 
\begin{defn}
A graph is \emph{chordal} if every cycle of length greater than three has a chord. 
A vertex of a graph is \emph{simplicial} if its neighbors induces a 
clique in the graph. 
\end{defn} 
One of the important properties of a chordal graph is that it has always a \emph{clique decomposition}. For the sake completeness, here 
we quickly review the idea of decomposing a chordal graph into cliques. For detailed information, one can refer to $[1]$. 
\\
\begin{defn}
For given graphs $G_{1}$ and $G_{2}$, we say that a graph $G$ arises from $G_{1}$ and $G_{2}$ \emph{by pasting along} $S$ 
if we have $G_{1}\cup G_{2}= G$ and $G_{1} \cap G_{2} = S$. In this case, the graphs $G_{i}$ are called the 
\emph{simplicial summands} of $G$.  
\end{defn}    

\begin{rem}
From the above definition, it is clear that a graph is chordal 
if it can be constructed recursively by pasting complete graphs along cliques. 
It is not hard to see that this process is independent of the order in which complete graphs
paste to each other. Indeed, this recursive construction gives us a clique decomposition of chordal 
graphs which is essential to obtain their clique polynomials. 
\end{rem}
For simplicity of arguments, we use the notation 
$G_{1}\cup_{S} G_{2}$ whenever $G_{1}$ and $G_{2}$
are pasted along $S$. The following lemma is key to 
obtain an explicit formula for the clique polynomial 
of chordal graphs. The proof is straight forward and left 
to the reader as a simple exercise.    
\begin{prop}
Let $G_1$ and $G_2$ be two simple graphs and 
$G = G_{1}\cup_{Q} G_{2}$ be their pasting along the
$i$ - clique $Q$. Then, we have 
\begin{equation}\label{keypast1}
C(G,x)
=
C(G_{1},x) + C(G_{2},x) - (x+1)^{i}, \hspace{1cm}
(i \geq 1).
\end{equation} 

\end{prop}
By the successive application of the formula
(\ref{keypast1})
and the 
recursive construction of chordal graphs, we can 
obtain the following explicit 
formula for the clique polynomial of chordal graphs. 
\begin{thm}
Let $G$ be a chordal graph defined as a pasting 
of the complete graphs 
$\{G_{i}\}^{r}_{i=1}$ of sizes $n_{i}$'s, respectively.
That is, 
$
G = 
G_{1}\cup_{Q_{1}} G_{2} \cup_{Q_{2}} \cdots
\cup_{Q_{r-1}} G_{r},
$ 
where 
$\{Q_{j}\}^{r-1}_{j=1}$ are cliques of sizes
$l_{j}$'s, respectively. Then, we have
\begin{equation}
C(G,x)
=
\sum_{i=1}^{r}(x+1)^{n_i} - 
\sum_{j=1}^{r-1}(x+1)^{l_j}
\end{equation}
  
\end{thm}
As an immediate consequence of 
the above theorem, we have the following 
interesting result. 
\begin{cor}
Any chordal graphs $G$ without isolated vertices 
has always a clique root $-1$
.
The multiplicity of this root is equal to the 
size of the smallest clique in the pasting process of the 
recursive construction of $G$.   
\end{cor}

\section{Proofs of the Main Results}
Here, we first give a proof of
the following proposition which is  
a weaker version of Theorem  
\ref{keythm1}.
From now on, for simplicity of arguments,
we will assume that our graphs are connected. 
\begin{prop}\label{keyprop1}
The class of 
$K_{4}$ - free planar chordal graphs 
has only clique roots. In particular, $-1$ is always a clique root. 
\end{prop}

\begin{proof}
For a given $K_{4}$ - free planar graph $G$, by \emph{Euler Formula}, we have
\begin{equation}\label{equt1}
n - m + f = 2,
\end{equation}
where $n$, $m$ and $f$ are the number of vertices, edges and faces of a 
\emph{planar embedding} of $G$, respectively. Moreover, if $G$ is a chordal 
graphs, then we observe that 
\begin{equation}\label{equt2}
f = t + 1,
\end{equation}
where $t$ is the number of triangles ( triangular faces ) of $G$. Hence, form 
(\ref{equt1}) and (\ref{equt2}), we conclude that
$$
n - m +  t = 1, 
$$
or equivalently
\begin{equation}
1 - n + m - t = 0,
\end{equation}
By last identity and considering the fact that the clique 
polynomial of a $K_{4}$ - free graph G is 
$C(G,x) = 1 + n x + m x^2 + t x^3$, we get 
$$
C(G,-1) = 1-n + m - t = 0.
$$
That is, the clique polynomial of any 
$K_{4}$ - free planar chordal graph $G$ has always 
$-1$ as 
a clique root. 
Therefore, we obtain the following multiplicative 
decomposition of $C(G,x)$
\begin{equation}
C(G,x)= (1 + x)( 1 + (n-1)x + (m - n +1 )x^2).
\end{equation}
The 
final step of the proof is to show that the 
quadratic 
polynomial:
\begin{equation}
Q(G,x) = 1 + (n-1)x + (m - n +1 )x^2,
\end{equation}
has always a real root. 
To this end, we actually prove that 
$Q(G,x)$ is a clique polynomial of a 
triangle - free graph 
$\tilde{G}$ which is obtained from 
the original graph $G$ based on the 
idea of a spanning tree of $G$. 
\\
For a given $K_{4}$ - free chordal graph $G$, 
pick up an arbitrary vertex $v \in V(G)$. 
Now, we construct a spanning tree of 
$G$ rooted at the vertex $v$. We will denote it by
$T_{G}$. Clearly, this tree has $n$ vertices and 
$n-1$ edges. 
Now in the graph $G$, delete all $n-1$ edges 
of the tree $T_{G}$ and call the resulting graph 
$\hat{G}$. This graph has clearly $n$ 
vertices and $m-(n-1)$ edges. Finally, by construction 
 $\hat{G}$ is triangle - free graph
and $v$ is an isolated vertex.   
Thus, the graph 
$\tilde{G}$ 
obtained from $\hat{G}$ by deleting the vertex 
$v$ has $n-1$ vertices and $m-n+1$ edges, as required.

\end{proof}
\begin{rem}
The following figure shows the process of 
obtaining the graph 
$\tilde{G}$ from the original graph $G$
for a sample graph $G$.

\begin{center}
\begin{tikzpicture}
\draw (0,0) -- (0.5,1) -- (1,0) -- (0,0);
\draw (0,0) -- (0.5,-1) -- (1,0) -- (1.5,1) 
-- (0.5,1);
\draw (1,0) -- (1.5,-1) -- (0.5,-1); 

\draw [fill] (0,0) circle [radius=0.07];
\draw [fill] (0.5,1) circle [radius=0.07];
\draw [fill] (1,0) circle [radius=0.07];
\draw [fill] (0.5,-1) circle [radius=0.07];
\draw [fill] (1.5,-1) circle [radius=0.07];
\draw [fill] (1.5,1) circle [radius=0.07];

\node [left] at (0,0) {$r$};
\node [left] at (2.3,0) {$\Rightarrow$};
\node [left] at (-0.5,0.35) {$G:$};

\draw [] (3.5,0) -- (4,1) ;
\draw [] (3.5,0) -- (4.5,0) ;
\draw [] (3.5,0) -- (4,-1) ;
\draw [] (4,1) -- (5,1) ; 
\draw [] (4.5,0) -- (5,-1) ;

\draw [fill] (3.5,0) circle [radius=0.07];
\draw [fill] (4,1) circle [radius=0.07];
\draw [fill] (4.5,0) circle [radius=0.07];
\draw [fill] (4,-1) circle [radius=0.07];
\draw [fill] (5,-1) circle [radius=0.07];
\draw [fill] (5,1) circle [radius=0.07];

\node [left] at (3.5,0) {$r$};
\node [left] at (5.7,0) {$\Rightarrow$};
\node [left] at (3.0,0.35) {$T_{G}:$};

\draw [] (7.5,1) -- (8,0) ;
\draw [] (8,0) -- (8.5,1) ;
\draw [] (8,0) -- (7.5,-1) ;
\draw [] (7.5,-1) -- (8.5,-1) ;

\draw [fill] (7,0) circle [radius=0.07];
\draw [fill] (7.5,1) circle [radius=0.07];
\draw [fill] (8,0) circle [radius=0.07];
\draw [fill] (7.5,-1) circle [radius=0.07];
\draw [fill] (8.5,-1) circle [radius=0.07];
\draw [fill] (8.5,1) circle [radius=0.07];

\node [left] at (7,0) {$r$};
\node [left] at (9,0) {$\Rightarrow$};
\node [left] at (6.3,0.35) {$\hat{G}:$};

\draw [] (10,1) -- (10.5,0) ;
\draw [] (10.5,0) -- (11,1) ;
\draw [] (10.5,0) -- (10,-1) ;
\draw [] (10,-1) -- (11,-1) ;

%\draw [fill] (7,0) circle [radius=0.07];
\draw [fill] (10,1) circle [radius=0.07];
\draw [fill] (10.5,0) circle [radius=0.07];
\draw [fill] (10,-1) circle [radius=0.07];
\draw [fill] (11,-1) circle [radius=0.07];
\draw [fill] (11,1) circle [radius=0.07];

\node [left] at (9.7,0.35) {$\tilde{G}:$};

\end{tikzpicture}

\end{center}
\begin{center}
\text{
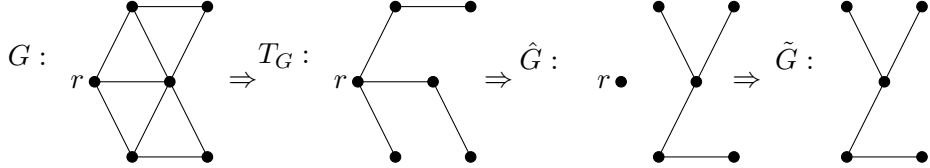
Fig1. The Construction of Triangle-Free graph $\tilde{G}$
}
\end{center}

\end{rem}
Now, we are ready to give a proof of our main 
theorem. 
\begin{proof}[Proof of Theorem \ref{keythm1}]
We first note that by Corollary 
\ref{keycor1}, any connected chordal graph has always a clique root 
$-1$. Now, the rest of the proof is exactly the same as 
the proof of Proposition \ref{keyprop1}. 
\end{proof}

Next we give an algebraic proof of \emph{Turan's Graph Theorem} $[3]$
for $K_{4}$ - free graphs which is indeed 
Corollary \ref{keycor1}. 

\begin{proof}[Proof of Corollary \ref{keycor1}]
We first note that since we want to prove an upper 
bound for the maximum possible number of edges, without 
loss of generality we assume that the graph $G$ is 
chordal. 
\\
As we saw in the proof of Theorem \ref{keythm1}, if $G$ is a given 
$K_{4}$ - free chordal graph with $n$ vertices and $m$ edges, 
then the following \emph{quadratic} equation
$$
Q(G,x) = 1 + (n-1)x + (m - n +1 )x^2,
$$
has only real zeros. Hence, it's \emph{discriminant} is 
nonnegative and therefore we have the following inequality:
$$
(n-1)^{2} - 4(m - n +1) \geq 0,
$$
which is equivalent to 
\begin{equation}\label{ineq1}
m \leq \left(\frac{n + 1}{2}\right)^{2} - 1.
\end{equation} 
On the other hand, we have the inequality 
\begin{equation}\label{ineq2} 
\left(\frac{n + 1}{2}\right)^{2} - 1 \leq \frac{n^{2}}{3},
\end{equation}
which is equivalent to the obvious inequality $ (n-3)^{2} \geq 0$. 
Thus, the inequalities (\ref{ineq1}) and (\ref{ineq2}) 
immediately implies the Turan's inequality 
for $K_{4}$ - free graphs.   
\end{proof}

\section{Open problems and questions}

We already showed that the class of connected $K_{4}$ - free chordal graphs has only 
real roots. Now, one might ask whether the class of connected $K_{5}$ - free chordal 
graphs has the same property or not. 

Unfortunately, this is not true in general. For example, 
the graph $K_{4}^{+}$ ( a complete graph $K_{4}$ plus one edge ) 
has only two clique roots. Indeed, we have 
$$
C(K_{4}^{+}) = 1 + 5x + 7x^2 + 4x^3 + x^4 = (1 + x)(1+ 4x + 3x^2 + x^3). 
$$
Since the cubic polynomial $\phi(x) = 1+ 4x + 3x^2 + x^3$ has the first derivative $\phi(x)' = 3(x+1)^2 + 1$ 
which is always \emph{positive}, by the first derivative criteria,  
$\phi(x) = 1+ 4x + 3x^2 + x^3$ has exactly \emph{one} real root. 
Thus, we come up with the following first open question. 
\\
{\bf Open Question 1}. 
Which subclasses of $K_{5}$ - free chordal 
graphs have only clique roots? 
\\

Recall that the class of $3$ - trees 
are those graphs 
which can be constructed \emph{recursively} by 
starting with a complete graph $K_{4}$, and then 
\emph{repeatedly} adding vertices in such a way that 
each added vertex has exactly \emph{three} neighbors 
that form a clique (triangle). 
\\

By the above definition, it is not hard to see that the class of 
$3$ - tress is a subclass of 
$K_{5}$ - free chordal 
graphs. Next, we come up with the following conjecture.  
\\
{\bf Conjecture 1}. 
The class of $3$ - trees has only clique roots.  
\\
Considering the fact that any connected chordal graph has 
a clique root $-1$ and the recursive definition of 
of chordal graphs, we made the following stronger conjecture. 
\\
{\bf Conjecture 2}.
The class of connected $K_{5}$ - free chordal 
graphs with the clique root $-1$ of multiplicity $2$ 
has only clique roots.


\begin{thebibliography}{9}
\bibitem{HajiMehrabadi(1998)}
             H. Hajiabolhassan and M. L. Mehrabadi,
             \textit{On clique polynomials},
             Australasian
             Journal of Combinatorics., 18 (1998), 313-316.
             
\bibitem{Teimoori2016}
Hossein Teimoori, 
\textit{Multiplicity of the Root $-1$ for Clique Polynomials of Chordal Graphs},
In Preparation.

\bibitem{Aigner}
M. Aigner, \textit{Tur\'{a}n's Graph Theorem},
 Journal of American Mathematical Monthly., 6 (1995),
808-816.

\end{thebibliography}
\end{document}